\documentclass[12pt, a4paper]{article}

\usepackage{t1enc}
\usepackage{amsmath,amsthm,amsfonts,amssymb,bbm,color}
\usepackage{graphicx,psfrag,subfigure,url}
\usepackage[hmargin=2cm,vmargin=2cm]{geometry}
\usepackage{hyperref}
\hypersetup{colorlinks=true}

\newcommand{\C}{\ensuremath{\mathbb{C}}}

\newcommand{\Cv}[1]{\mathcal{C}_{#1}}
\newcommand{\Cs}[1]{\mathcal{D}_{#1}}

\numberwithin{equation}{section}
\newcommand{\I}{{\rm i}}
\newcommand{\Id}{\mathbbm{1}}

\newcommand{\N}{\mathbb{N}}

\newcommand{\EE}{\ensuremath{\mathbb{E}}}
\newcommand{\PP}{\ensuremath{\mathbb{P}}}
\newcommand{\R}{\ensuremath{\mathbb{R}}}
\newcommand{\Z}{\ensuremath{\mathbb{Z}}}
\newcommand{\Rplus}{\ensuremath{\mathbb{R}_{+}}}

\newcommand{\Zsd}{\ensuremath{\mathbf{Z}}}
\newcommand{\Zsdaalpha}{\ensuremath{\mathbf{Z}^{a,\alpha}}}
\newcommand{\Zsdbbeta}{\ensuremath{\mathbf{Z}^{b,\beta}}}
\newcommand{\Zc}{\ensuremath{\mathcal{Z}}}
\newcommand{\Zbbeta}{\ensuremath{\mathcal{Z}^{b,\beta}}}

\newcommand{\Fc}{\ensuremath{\mathcal{F}}}
\newcommand{\Fbbeta}{\ensuremath{\mathcal{F}^{b,\beta}}}
\newcommand{\Fsigbbeta}{\ensuremath{\mathcal{F}^{\sigma b,\sigma\beta}}}
\newcommand{\Fsdaalpha}{\ensuremath{\mathbf{F}^{a,\alpha}}}

\newcommand{\K}[1]{\mathrm{K}_{#1}}
\newcommand{\Kbbeta}{\mathrm{K}_{b,\beta}^{(\sigma)}}
\newcommand{\Ksigbbeta}{\mathrm{K}_{\sigma b,\sigma\beta}^{(\sigma)}}

\newcommand{\BP}{\mathrm{BP}}

\newcommand{\fredholmp}[2]{\det\left(\Id+#1\right)_{#2}}
\newcommand{\fredholmm}[2]{\det\left(\Id-#1\right)_{#2}}
\newcommand{\LC}{L^2(\Cv{a;\alpha;\varphi})}
\newcommand{\LRplus}{L^2(\Rplus)}
\newcommand{\Lrinf}{L^2(r,\infty)}

\newcommand{\ordo}{\mathcal{O}}

\newcommand{\wt}{\widetilde}

\renewcommand{\d}{\mathrm d}
\renewcommand{\Re}{\operatorname{Re}}
\renewcommand{\Im}{\operatorname{Im}}

\newtheorem{theorem}{Theorem}[section]
\newtheorem{proposition}[theorem]{Proposition}
\newtheorem{lemma}[theorem]{Lemma}

\newtheorem{definition}[theorem]{Definition}

\title{Borodin--P\'ech\'e fluctuations of the free energy in directed random polymer models}
\author{Zs\'ofia Talyig\'as\thanks{Department of Mathematical Sciences, University of Bath, Claverton Down, Bath, UK. E-mail: {\tt z.talyigas@bath.ac.uk}}
\and B\'alint Vet\H o\thanks{Institute of Mathematics, Budapest University of Technology and Economics, Egry J.\ u.\ 1, 1111 Budapest, Hungary. E-mail: {\tt vetob@math.bme.hu}}}
\date{}
\begin{document}
\maketitle
\begin{abstract}
We consider two directed polymer models in the Kardar--Parisi--Zhang (KPZ) universality class:
the O'Connell--Yor semi-discrete directed polymer with boundary sources and the continuum directed random polymer with $(m,n)$-spiked boundary perturbations.
The free energy of the continuum polymer is the Hopf--Cole solution of the KPZ equation with the corresponding $(m,n)$-spiked initial condition.
This new initial condition is constructed using two semi-discrete polymer models with independent bulk randomness and coupled boundary sources.
We prove that the limiting fluctuations of the free energies rescaled by the $1/3$rd power of time in both polymer models
converge to the Borodin--P\'ech\'e type deformations of the GUE Tracy--Widom distribution.
\end{abstract}

\section{Introduction}
\label{Intro}

The Kardar--Parisi--Zhang (KPZ) equation was introduced for the description of physical surface growth phenomena in~\cite{KPZ86}.
The equation gives the stochastic evolution of the height function $\Fc(T,X)$ where $T \in \Rplus$ is the time and $X\in\R$ is the space variable.
It reads as
\begin{equation}\label{kpz}
\partial_t \Fc(T,X) = \frac12 \partial_X^2 \Fc(T,X) + \frac12 (\partial_X \Fc(T,X))^2 + \xi(T,X), \qquad \Fc(0,X) = \Fc_0(X)
\end{equation}
where $\xi$ denotes space-time Gaussian white noise with $\EE\left[ \xi(T,X)\xi(S,Y) \right] = \delta(T-S)\delta(X-Y)$.
By the presence of the non-linear term, the equation is not rigorously well-posed and serious work is required to make sense of the solution directly~\cite{Hai11}.
A natural way to give a solution to the equation formally is via the stochastic heat equation (SHE) with multiplicative noise
\begin{equation}\label{she}
\partial_T \Zc(T,X)=\tfrac12\partial_X^2 \Zc(T,X)+\Zc(T,X)\xi(T,X), \qquad \Zc(0,X) = \Zc_0(X).
\end{equation}
The latter equation is well-posed and $\Fc(T,X)=\ln\mathcal Z(T,X)$ with initial condition $\Fc(0,X)=\ln\Zc(0,X)$ defines a formal solution to \eqref{kpz} which is the Hopf--Cole solution of the KPZ equation.
See~\cite{Cor11} for a review on the KPZ equation and its universality class which is the family of models with the same scaling and asymptotic behaviour as the solution of the KPZ equation.

The Hopf--Cole solution of the KPZ equation can be understood as the partition function of a directed polymer model by the Feynman--Kac representation
\begin{equation}\label{feynman}
\Zc(T,X) = \EE_{B(X)}\left[ \Zc_0(B(0)) : \exp : \left\{ \int_{0}^{T}\xi(t,B(t))dt \right\} \right]
\end{equation}
where the expectation $\EE$ is taken over the law of a Brownian motion $B$ which is running backwards from time $T$ and position $X$ and where $: \exp :$ is the Wick exponential.
The representation \eqref{feynman} defines the partition function of the continuum directed random polymer (CDRP)
as it is the total weight of Brownian paths where the weight is proportional to the exponential function of the integral of the disorder along the path.
The logarithm of the partition function $\Fc(T,X)=\ln\Zc(T,X)$ is called the free energy of the CDRP.

The present paper describes limiting fluctuations in two directed polymer models.
Directed polymers are well-studied objects in the KPZ universality class of models in the recent mathematics and physics literature.
The reason for the special interest is that certain models possess exact solvable properties, i.e.\ explicit formulas can be derived for some of their important observables.
The first directed polymer model with exact solvability is the O'Connell--Yor semi-discrete polymer~\cite{OCY01,OCon09}.
Exactly solvable polymers on the square lattice are the log-gamma polymer~\cite{Sep09,COSZ11,LeDoussalThierry}, the strict-weak polymer~\cite{CSS15,OCO15},
the beta polymer~\cite{BarCor15} and the inverse beta polymer~\cite{ThiDou15}.
Methods to obtain exact solvability include explicit stationary measure, Bethe Ansatz integrability and the (geometric) Robinson--Schensted--Knuth (RSK) correspondence.

In~\cite{BCF12}, the O'Connell--Yor model was considered with boundary perturbations.
The large time limit of the free energy was proved to be the Baik--Ben Arous--P\'ech\'e (BBP) distribution~\cite{BBP06} which is the perturbed version of the GUE Tracy--Widom distribution.
A similar limit distribution was obtained for the CDRP with $m$-spiked boundary perturbation in~\cite{BCF12}.

The results of the present paper generalize those of~\cite{BCF12} in the following sense.
We investigate the large scale behaviour of the free energy of two directed polymer models.
The first model is the O'Connell--Yor semi-discrete random polymer with log-gamma boundary sources~\cite{BCFV15}
which is the mixture of the O'Connell--Yor semi-discrete polymer with boundary perturbations considered in~\cite{BCF12} and the log-gamma discrete directed polymer.
Explicit Fredholm determinant expressions are available in~\cite{BCFV15} for the Laplace transform of the partition function of the polymer mixture model.
Based on these formulas, we obtain the single time version of the Borodin--P\'ech\'e distribution as the limit distribution of the free energy.
The Borodin--P\'ech\'e distribution which is a generalization of the BBP distribution was first described in its multi-time version in last passage percolation with defective rows and columns
and in a single time version in a random matrix model in~\cite{BorodinPeche08}.

A closely related model is the stationary O'Connell--Yor polymer model which was considered in~\cite{IS17} as the limit of the O'Connell--Yor semi-discrete polymer model with log-gamma boundary sources.
It was proved in~\cite{IS17} that the large time limit of the stationary model is the Baik--Rains distribution
and the solution of the stationary KPZ equation was obtained as the scaling limit of the stationary O'Connell--Yor polymer.

The second model considered in the present paper is the CDRP which can be obtained as the limit of the O'Connell--Yor semi-discrete polymer under the intermediate disorder scaling~\cite{QMR12,Nic16}.
Extending the investigations of the CDRP with $m$-spiked boundary perturbation in~\cite{BCF12}, we introduce the $(m,n)$-spiked boundary perturbation.
The $m$-spiked boundary perturbation is non-zero for the positive values of the space variable,
the $(m,n)$-spiked boundary perturbation can be seen as its two-sided version with the appropriate coupling of the two sides.
We prove Borodin--P\'ech\'e limit distribution for the free energy of the CDRP with $(m,n)$-spiked boundary perturbation based on explicit Fredholm determinant formulas from~\cite{BCFV15}.

The rest of the paper is organized as follows.
We introduce the O'Connell--Yor semi-discrete directed random polymer with log-gamma boundary sources and the CDRP with $(m,n)$-spiked boundary perturbation in Section~\ref{s:modelsresults}.
Our main results, Theorem~\ref{ThmFreeEnergyScalingLimit} and Theorem~\ref{ThmTimeLimitGen} are also stated in this section.
We prove Theorem~\ref{ThmFreeEnergyScalingLimit} in Section~\ref{s:OYlimit} and Theorem~\ref{ThmTimeLimitGen} in Section~\ref{s:CDRPlimit}.

\paragraph{Acknowledgements.}
We thank Patrik Ferrari and Benedek Valk\'o for stimulating discussions related to this project.
The work of both authors was supported by the NKFI (National Research, Development and Innovation Office) grant FK123962.
The work of B.\ Vet\H o was supported by the NKFI grant PD123994, by the Bolyai Research Scholarship of the Hungarian Academy of Sciences
and by the \'UNKP--18--4 New National Excellence Program of the Ministry of Human Capacities.

\section{Models and main results}
\label{s:modelsresults}

We present the two models considered in this paper: the O'Connell--Yor semi-discrete directed polymer with log-gamma boundary sources
and the continuum directed random polymer (CDRP) with $(m,n)$-spiked boundary perturbation.
These models were defined in~\cite{BCFV15}, but the $(m,n)$-spiked boundary perturbation is new.
We consider a slightly different scaling of the boundary perturbations as in~\cite{BCFV15} yielding our main results which are stated in this section.

\subsection{O'Connell--Yor semi-discrete directed polymer with log-gamma boundary sources}

The O'Connell--Yor semi-discrete directed polymer with log-gamma boundary sources is the mixture of the semi-discrete polymer model introduced by O'Connell and Yor~\cite{OCY01}
and the discrete one by Sepp\"{a}l\"{a}inen~\cite{Sep09}.
By log-gamma distribution with parameter $\theta>0$ we mean the distribution of the random variable $-\ln X$
where $X$ has gamma distribution with parameter $\theta$, i.e.\ when $X$ has density $x^{\theta-1}e^{-x}/\Gamma(\theta)$ for $x>0$.

Fix $N\geq 1$ and $n\geq 0$. Let $a=(a_1,\ldots,a_N)\in \R^N$ and $\alpha = (\alpha_1,\ldots,\alpha_n)\in\R_+^n$ be such that $\alpha_k-a_l>0$ for all $1\leq l\leq N$ and $1\leq k\leq n$.
In the polymer model that we introduce, the horizontal axis is discrete on the left of $0$ and continuous on the right of $0$ while the vertical axis is discrete.
For all $1\leq k\leq n$ and $1\leq l\leq N$, let $\omega_{-k,l}$ be independent log-gamma random variables with parameter $\alpha_k-a_l$.
For all $1\leq l\leq N$, let $B_l$ be independent Brownian motions with drift $a_l$ which are also independent of the log-gamma variables.
The $\omega_{-k,l}$ can be thought of as sitting at the lattice points $(-k,l)$ while $B_l$ can be thought of as sitting along the horizontal ray from $(0,l)$ as shown on Figure~\ref{FigSemiDiscreteDP}.

\begin{figure}\begin{center}
\psfrag{w11}[cc]{$\omega_{-1,1}$}
\psfrag{w21}[cc]{$\omega_{-2,1}$}
\psfrag{wM1}[cc]{$\omega_{-n,1}$}
\psfrag{w1N}[bc]{$\,\omega_{-1,N}$}
\psfrag{w2N}[bc]{$\!\omega_{-2,N}$}
\psfrag{wMN}[bc]{$\omega_{-n,N}$}
\psfrag{tn}[bc]{$(\tau,N)$}
\psfrag{B1}[lc]{$B_1$}
\psfrag{B2}[lc]{$B_2$}
\psfrag{B3}[lc]{$B_3$}
\psfrag{BN}[lc]{$B_N$}
\psfrag{s2}[cc]{$s_2$}
\psfrag{s3}[cc]{$s_3$}
\psfrag{sNm1}[cc]{$s_{N-1}$}
\psfrag{t}[cc]{$\tau$}
\psfrag{phi}[bc]{$\phi$}
\includegraphics[height=5cm]{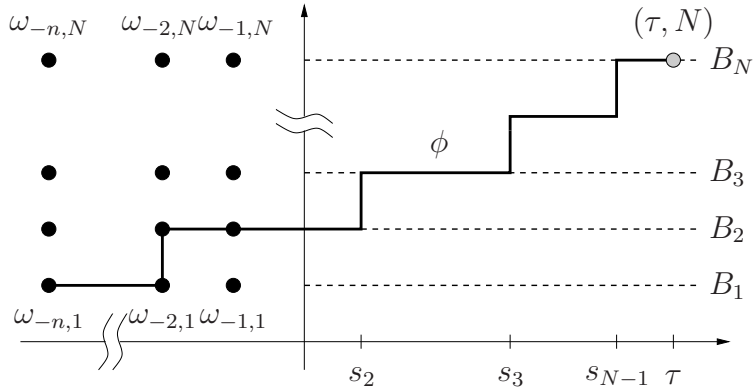}
\caption{The O'Connell--Yor semi-discrete directed polymer with log-gamma boundary sources.
The thick solid line is a possible path $\phi$ from $(-n,1)$ to $(\tau,N)$.
The random variables $\omega_{-k,l}$ have log-gamma distribution with parameter $\alpha_k-a_l$ and the Brownian motions $B_1,\ldots,B_N$ have drifts $a_1,\ldots,a_N$.}
\label{FigSemiDiscreteDP}
\end{center}\end{figure}

Admissible paths consist of discrete and semi-discrete parts.
A discrete up-right path $\phi^d:(i_1,j_1)\nearrow(i_{\ell},j_{\ell})$ is an ordered set of points
$((i_1,j_1),(i_2,j_2),\ldots, (i_{\ell},j_{\ell}))$ with each $(i_k,j_k)\in \Z^2$ and each increment $(i_k,j_k) -(i_{k-1},j_{k-1})$ either $(1,0)$ or $(0,1)$.
A semi-discrete up-right path $\phi^{sd}:(0,l)\nearrow (\tau,N)$ is a union of horizontal line segments
$((0,l)\to (s_l,l))\cup((s_l,l+1)\to (s_{l+1},l+1))\cup\cdots\cup((s_{N-1},N)\to (\tau,N))$ where $0\leq s_l<s_{l+1}<\cdots <s_{N-1}\leq \tau$.
It is convenient to think of $\phi^{sd}$ as a surjective non-decreasing function from $[0,\tau]$ onto $\{l,\dots, N\}$.
Our up-right paths $\phi$ in the mixture model are composed of discrete portions $\phi^d$ adjoined to semi-discrete portions $\phi^{sd}$
in such a way that for some $1\leq l\leq N$, $\phi^d:(-n,1)\nearrow (-1,l)$ and $\phi^{sd}:(0,l)\nearrow (\tau,N)$. 

To an up-right path described above, we associate an energy
\begin{equation}\label{eq:Energy}
E(\phi)=\sum_{(i,j)\in \phi^d} \omega_{i,j} + B_l(s_l)+(B_{l+1}(s_{l+1})-B_{l+1}(s_l))+\dots+(B_N(\tau)-B_N(s_{N-1}))
\end{equation}
which aggregates the randomness along the path, hence itself is random depending on $\omega_{i,j}$ and $B_1,\dots,B_N$.
The polymer measure on a path $\phi$ is proportional to its Boltzmann weight given by $e^{E(\phi)}$.
The normalizing constant or polymer partition function for the O'Connell--Yor semi-discrete directed polymer with log-gamma boundary sources
is the integral of the Boltzmann weight over the background measure on the path space $\phi$, i.e.
\begin{equation}\label{defZOY}
\Zsdaalpha(\tau,N) = \sum_{l=1}^{N} \, \sum_{\phi^d:(-n,1)\nearrow (-1,l)} \, \int_{\phi^{sd}:(0,l)\nearrow(\tau,N)} e^{E(\phi)} \, \d\phi^{sd}
\end{equation}
where $\d\phi^{sd}$ represents the Lebesgue measure on the simplex $0\leq s_k<s_{k+1}<\dots <s_{N-1}\leq \tau$ with which $\phi^{sd}$ is identified.
The free energy of the O'Connell--Yor semi-discrete directed polymer with log-gamma boundary sources is given by
\begin{equation}\label{defF}
\Fsdaalpha(\tau,N) = \ln\left(\Zsdaalpha(\tau,N)\right).
\end{equation}
The distribution of the partition function $\Zsdaalpha(\tau,N)$ of the O'Connell--Yor semi-discrete directed polymer with log-gamma boundary sources was characterized in~\cite{BCFV15} as follows.

\begin{theorem}\cite[Theorem 2.1]{BCFV15}\label{ThmFormulaSemiDiscrete}
Fix $N\geq 9$, $n\geq 0$ and $\tau> 0$.
Let $a=(a_1,\ldots,a_N)\in \R^N$ and $\alpha = (\alpha_1,\ldots,\alpha_n)\in \R_+^n$ be such that $\alpha_k-a_l>0$ for all $1\leq l\leq N$ and $1\leq k\leq n$.
For $1\leq k\leq n$ and $1\leq l\leq N$ let $\omega_{-k,l}$ be independent log-gamma random variables with parameter $\alpha_k-a_l$
and for all $1\leq l\leq N$ let $B_l$ be independent Brownian motions with drift $a_l$.
Then for all $u\in \C$ with positive real part
\begin{equation}\label{semidiscreteFredholm}
\EE\left( e^{-u \Zsdaalpha(\tau,N)} \right) = \fredholmp{\K{u}}{\LC}
\end{equation}
where the operator $\K{u}$ is defined in terms of its integral kernel
\begin{equation}\label{defKu}
\K{u}(v,v')=\frac{1}{2\pi \I}\int_{\Cs{v}}\d s\, \Gamma(-s)\Gamma(1+s)
\frac{ u^s e^{v\tau s+\tau s^2/2}}{v+s-v'} \prod_{l=1}^{N}\frac{\Gamma(v-a_l)}{\Gamma(s+v-a_l)} \prod_{k=1}^n\frac{\Gamma(\alpha_k-v-s)}{\Gamma(\alpha_k-v)}.
\end{equation}
The contours $\Cv{a;\alpha;\varphi}$ and $\Cs{v}$ are given in Definition~\ref{DefCaCsdefBis} below where $\varphi\in(0,\pi/4)$ is arbitrary.
\end{theorem}

\begin{definition}\label{DefCaCsdefBis}
Let $a=(a_1,\ldots,a_N)\in \R^N$ and $\alpha = (\alpha_1,\ldots,\alpha_n)\in \R_+^n$ be such that $\alpha_k-a_l>0$ for all $1\leq l\leq N$ and $1\leq k\leq n$.
Set $\mu=\tfrac{1}{2}\max(a)+\tfrac{1}{2}\min(\alpha)$ and $\eta = \tfrac{1}{4}\max(a)+\tfrac{3}{4}\min(\alpha)$.
Then, for all $\varphi\in (0,\pi/4)$, we define the contour $\Cv{a;\alpha;\varphi}=\{\mu+e^{\I (\pi+\varphi)}y\}_{y\in \Rplus}\cup \{\mu+e^{\I(\pi-\varphi)}y\}_{y\in \Rplus}$.
The contour is oriented so as to have increasing imaginary part.
For every $v\in \Cv{a;\alpha;\varphi}$, we choose $R=-\Re(v)+\eta$, $d>0$, and define a contour $\Cs{v}$ as follows.
$\Cs{v}$ goes by straight lines from $R-\I \infty$, to $R-\I d$, to $1/2-\I d$, to $1/2+\I d$, to $R+\I d$, to $R+\I\infty$.
The parameter $d$ is taken small enough so that $v+\Cs{v}$ does not intersect $\Cv{a;\alpha;\varphi}$.
See Figure~\ref{FigContoursSemiDiscrete} for an illustration.
\end{definition}

\begin{figure}\begin{center}
\psfrag{Cv}[lb]{$\Cv{a;\alpha;\varphi}$}
\psfrag{v+Cs}[lb]{$v+\Cs{v}$}
\psfrag{Cs}[lb]{$\Cs{v}$}
\psfrag{mu}[cb]{$\mu$}
\psfrag{eta}[cb]{$\eta$}
\psfrag{v}[cb]{$v$}
\psfrag{R}[cb]{$R$}
\psfrag{2d}[lb]{$2d$}
\psfrag{alpha}[cb]{$\alpha$'s}
\psfrag{a}[cb]{$a$'s}
\psfrag{0}[cb]{$0$}
\psfrag{phi}[lb]{$\varphi$}
\includegraphics[height=5cm]{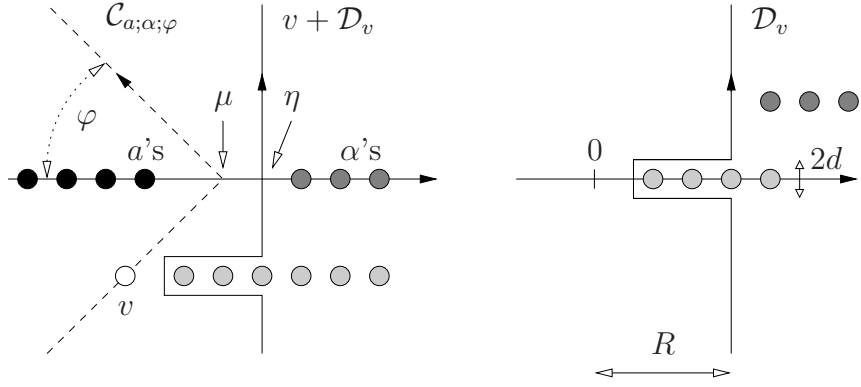}
\caption{Left: the contour $\Cv{a;\alpha;\varphi}$ (dashed) where the black dots symbolize the set of singularities of $K_u(v,v')$ in $v$ at $\cup_{1\leq l\leq N}\{a_l,a_l-1,\dots\}$
coming from the factors $\Gamma(v-a_l)$.
The contour $v+\Cs{v}$ is the solid line.
Right: the contour $\Cs{v}$ where the light grey dots are the singularities at $\{1,2,\dots\}$ coming from $\Gamma(-s)$
and the dark grey dots are those at $\cup_{1\leq k \leq n}\{\alpha_k-v,\alpha_k+1-v,\dots\}$ coming from $\Gamma(\alpha_k-v-s)$.}
\label{FigContoursSemiDiscrete}
\end{center}\end{figure}

Our contribution on the O'Connell--Yor semi-discrete directed polymer with log-gamma boundary sources is that we prove a Borodin--P\'ech\'e scaling limit of its free energy.
To define the limiting distribution, fix two integers $m$ and $n$.
Let $b=(b_1,\ldots,b_m)\in\R^m$ and $\beta=(\beta_1,\ldots,\beta_n)\in\R^n$ be two sets of parameters and assume that
\begin{equation}\label{bbetacond}
\max_{1\le l\le m}b_l<\min_{1\le k\le n}\beta_k
\end{equation}
which is a natural constraint, since otherwise the corresponding polymer models are not well defined, see Theorem~\ref{ThmFreeEnergyScalingLimit} and \ref{ThmTimeLimitGen} below for the parameter scaling.

The Borodin--P\'ech\'e distribution with parameters $b$ and $\beta$ \cite{BorodinPeche08} is defined as
\begin{equation}\label{defBPdistr}
F_{\BP,b,\beta}(r) = \fredholmm{\K{\BP,b,\beta}}{L^2((r,\infty))}
\end{equation}
with the kernel
\begin{equation}\label{defBPkernel}
\K{\BP,b,\beta}(x,y) = \frac1{(2\pi\I)^2}\int_\gamma\d w\int_\Gamma\d z \frac{1}{z-w} \frac{e^{z^3/3-zy}}{e^{w^3/3-wx}}\prod_{l=1}^{m}\frac{z-b_l}{w-b_l}\prod_{k=1}^n\frac{w-\beta_k}{z-\beta_k}
\end{equation}
where the integration contours $\gamma$ and $\Gamma$ are given as follows.
Let $c>0$ be arbitrary.
Then $\gamma$ is $-c+\I\R$ modified in a neighbourhood of the real axis so that it crosses the axis between $\max_{1\le l\le m}b_l$ and $\min_{1\le k\le n}\beta_k$.
The contour $\Gamma$ is $c+\I\R$ modified in a neighbourhood of the real axis so that it crosses the real axis between $\max_{1\le l\le m}b_l$ and $\min_{1\le k\le n}\beta_k$
and it does not intersect $\gamma$.
We mention that for $n=0$, the Borodin--P\'ech\'e distribution reduces to the BBP distribution and for $n=m=0$ to the GUE Tracy--Widom distribution.

To state our main theorem on the scaling limit of the O'Connell--Yor semi-discrete directed polymer with log-gamma boundary sources, we will use the following parametrization.
Let $\Psi(z) = \frac{d}{dz}\ln \Gamma(z)$ be the digamma function.
For a given $\theta\in \Rplus$, define
\begin{equation}\label{parametr1}
\kappa(\theta)=\Psi'(\theta),\quad f(\theta)=\theta\Psi'(\theta)-\Psi(\theta),\quad c(\theta)=(-\Psi''(\theta)/2)^{1/3}.
\end{equation}
We may alternatively parameterize $\theta\in\Rplus$ in terms of $\kappa\in\Rplus$ as
\begin{equation}\label{parametr2}
\theta_\kappa=(\Psi')^{-1}(\kappa)\in\Rplus,\quad f_\kappa=\inf_{t>0} (\kappa t - \Psi(t))=f(\theta_\kappa),\quad c_\kappa=c(\theta_\kappa).
\end{equation}

\begin{theorem}\label{ThmFreeEnergyScalingLimit}
Consider the O'Connell--Yor semi-discrete directed random polymer with log-gamma boundary sources of the following parameters.
Let $a = (a_1, a_2, \dots ,a_m, 0, \dots 0)\in\R^N$ with $m \leq N$ and $\alpha = (\alpha_1,\alpha_2,\dots,\alpha_n)\in\R^n$ where $a_1,a_2,\dots a_m$ may depend on $N$
and $\alpha_k > \max_{1\leq l \leq m}a_l$ for $1\leq k \leq n$.
Let $\kappa>0$ be arbitrary.
Assume furthermore that there are real parameters $b=(b_1,\dots, b_m)$ and $\beta = (\beta_1, \dots \beta_n)$ satisfying \eqref{bbetacond} such that for any $1\le l\le m$ and $1\le k\le n$,
\begin{equation}\label{aalphascaling}
\lim_{N\to\infty} c_{\kappa}N^{1/3}(a_l(N) - \theta_{\kappa}) = b_l \quad\mbox{and}\quad\lim_{N\to\infty} c_{\kappa}N^{1/3}(\alpha_k(N) - \theta_{\kappa}) = \beta_k.
\end{equation}
Then
\begin{equation}
\lim_{N\to\infty} \PP \left( \frac{\Fsdaalpha(\kappa N,N) - N f_{\kappa}}{c_{\kappa}N^{1/3}} \leq r \right) = F_{\BP,b,\beta}(r)
\end{equation}
holds where $\Fsdaalpha$ is the free energy of the O'Connell--Yor semi-discrete directed random polymer given in \eqref{defF}
and $F_{\BP,b,\beta}$ is the Borodin--P\'ech\'e distribution function defined in \eqref{defBPdistr}.
\end{theorem}

\subsection{Continuum directed random polymer (CDRP) with $(m,n)$-spiked boundary perturbation}\label{SectCDRP}

The partition function $\Zc(T,X)$ of the continuum directed random polymer with boundary perturbation $\Zc_0(X)$
is given by the solution to the stochastic heat equation with multiplicative noise \eqref{she} with initial condition $\Zc_0(X)$.
The initial data $\Zc_0(X)$ may be random but it is assumed to be independent of the space-time white noise.

By the Feynman--Kac representation \eqref{feynman}, $\Zc(T,X)$ is indeed a partition function of a directed polymer model,
since Brownian paths are reweighted in a way that the weight of a path is proportional to the Wick exponential of the randomness integrated along the path.
The normalizing constant which is the partition function $\Zc(T,X)$ is the integral of weights over the space of all possible paths.
Note that $\Zc(T,X)$ itself is random as the randomness of the space-time white noise remains in the formula \eqref{feynman}.

By the work of Mueller~\cite{Mue91}, as long as $\Zc_0(X)$ is almost surely positive, $\Zc(T,X)$ is positive for all $T > 0$ and $X\in\R$ almost surely.
Hence we can take its logarithm and define the free energy for the continuum directed random polymer with boundary perturbation $\ln\Zc_0(X)$ by
$\Fc(T,X)=\ln(\Zc(T,X))$ to be the Hopf--Cole solution of the KPZ equation \eqref{kpz} with initial condition $\Fc_0(X)=\ln\Zc_0(X)$.

Let us now introduce the CDRP with $(m,n)$-spiked boundary perturbation and let us construct the corresponding $(m,n)$-spiked initial condition for the stochastic heat equation.
For fixed integers $m$ and $n$, let $b=(b_1,\ldots,b_m)\in\R^m$ and $\beta=(\beta_1,\ldots,\beta_n)\in\R^n$ be such that \eqref{bbetacond} holds.
Let $B_1,B_2,\dots,B_m$ be independent Brownian motions with drifts $b_1,b_2,\dots,b_m$,
and let $\wt B_1,\wt B_2,\dots,\wt B_n$ be independent Brownian motions with drifts $\beta_1,\beta_2,\dots,\beta_n$.
Furthermore, let $\omega_{-k,l}$ be independent log-gamma random variables with parameter $\beta_k-b_l$ for $1\leq l\leq m$ and $1\leq k\leq n$.
Assume that the two families of Brownian motions and the log-gamma random variables are independent of each other.
For $X\ge0$, let the semi-discrete partition function $\Zsdbbeta(X,m)$ be constructed as in \eqref{defZOY} using the Brownian motions $B_1,B_2,\dots,B_m$ and the log-gamma random variables.

Similarly, we construct another semi-discrete partition function which is coupled to the previous one.
Let the possible paths $\wt\phi$ be composed of a discrete up-right part $\wt\phi^d:(-n,1)\nearrow(k-n-1,m)$ and of a semi-discrete part $\wt\phi^{sd}$.
For $\wt X\ge0$, let the semi-discrete part $\wt\phi^{sd}:(k-n-1,m)\nearrow(-1,\wt X)$ be a union of vertical line segments
$((k-n-1,m)\to(k-n-1,s_k))\cup((k-n,s_k)\to(k-n,s_{k+1}))\cup\dots\cup((-1,s_{n-1})\to(-1,\wt X))$ where $m\le s_k<s_{k+1}<\dots<s_{n-1}\le\wt X$.
The energy of such a path is instead of \eqref{eq:Energy} defined by
\begin{equation}\label{eq:wtEnergy}
E\left(\wt\phi\right)=\sum_{(i,j)\in\wt\phi^d} \omega_{i,j} + \wt B_k(s_k)+(\wt B_{k+1}(s_{k+1})-\wt B_{k+1}(s_k))+\dots+(\wt B_n(\wt X)-\wt B_n(s_{n-1})).
\end{equation}
Then a partition function analogously to \eqref{defZOY} is given by
\begin{equation}\label{defwtZOY}
\wt\Zsd^{b,\beta}(\wt X,n) = \sum_{k=1}^{n} \, \sum_{\wt\phi^d:(-n,1)\nearrow (k-n-1,m)} \, \int_{\wt\phi^{sd}:(k-n-1,m)\nearrow(-1,\wt X)} e^{E\left(\wt\phi\right)} \, \d\wt\phi^{sd}.
\end{equation}
The Brownian motions $\wt B_k$ can be thought of as sitting on the vertical rays starting at $(k-n-1,m)$ for $1\le k\le n$
which makes the definitions \eqref{eq:wtEnergy}--\eqref{defwtZOY} natural.

\begin{figure}\begin{center}
\psfrag{w11}{$\omega_{-1,1}$}
\psfrag{w21}{$\omega_{-2,1}$}
\psfrag{wn-11}{$\,\omega_{-n+1,1}$}
\psfrag{wn1}{$\!\omega_{-n,1}$}
\psfrag{Bm}[lc]{$B_m$}
\psfrag{B1}[lc]{$B_1$}
\psfrag{B2}[lc]{$B_2$}
\psfrag{Bnt}[lc]{$\wt{B}_n$}
\psfrag{Bn-1t}[lc]{$\wt{B}_{n-1}$}
\psfrag{B2t}[lc]{$\wt{B}_2$}
\psfrag{B1t}[lc]{$\wt{B}_1$}
\psfrag{phl}[bc]{$\phi_{\textrm{left}}$}
\psfrag{phr}[bc]{$\phi_{\textrm{right}}$}
\includegraphics[height=6cm]{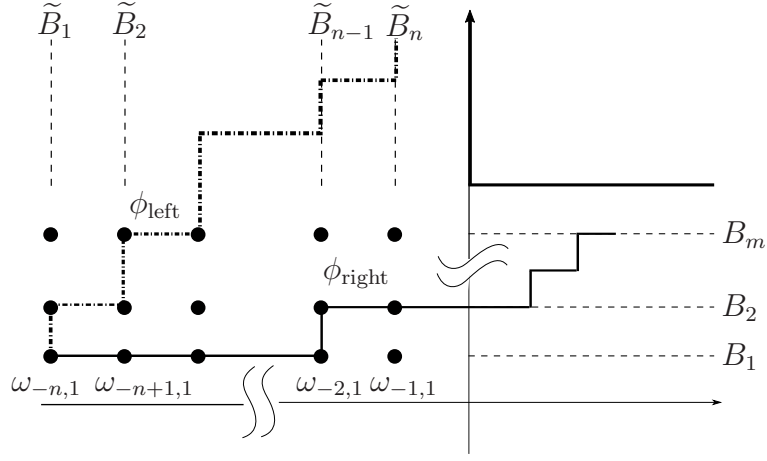}
\caption{The $(m,n)$-spiked boundary perturbation $\Zbbeta_0$ for the CDRP, i.e.\ the $(m,n)$-spiked initial condition for the stochastic heat equation.
It is realized by two semi-discrete polymer partition functions with log-gamma boundary sources where the log-gamma random variables are sampled jointly.
For $X>0$, $\Zsdbbeta(X,m)$ appears on the horizontal half-line starting at $(0,m+1)$ whereas $\wt{\Zsd}^{\beta,b}(-X,n)$ for $X\le0$ appears on the vertical half-line starting at the same point.
\label{fig:mnspikedIC}}
\end{center}\end{figure}

For $b\in\R^m$ and $\beta\in\R^n$, let
\begin{equation}\label{eq:SHEIC}
\Zbbeta_0(X) = \left\{\begin{array}{ll}\Zsdbbeta(X,m), & \mbox{if } X > 0,\\ \wt{\Zsd}^{\beta,b}(-X,n), & \mbox{if } X\leq 0\end{array}\right.
\end{equation}
define the $(m,n)$-spiked boundary perturbation for the CDRP, see Figure~\ref{fig:mnspikedIC}.
Let $\Zbbeta(T,X)$ denote the partition function of the CDRP with $(m,n)$-spiked boundary perturbation
which is the solution of the stochastic heat equation \eqref{she} with initial condition given by \eqref{eq:SHEIC}.
Let $\Fbbeta(T,X) = \ln(\Zbbeta(T,X))$ denote the free energy of the CDRP with $(m,n)$-spiked boundary perturbation.
Note that for $n=0$, the boundary perturbation \eqref{eq:SHEIC} reduces to the $m$-spiked boundary perturbation considered in~\cite{BCF12}.

According to the next theorem, the CDRP with $(m,n)$-spiked boundary perturbations is the limit of the O'Connell--Yor semi-discrete directed polymer with log-gamma boundary sources
under the intermediate disorder scaling.
The theorem in this form was not published yet, it was first announced in~\cite{QMR12} for the O'Connell--Yor semi-discrete directed polymer with boundary perturbations and used e.g.\ in~\cite{BCF12,BCFV15}.
Theorem~\ref{ThmDiscrToContGen} for perturbed boundaries below is a straightforward consequence of the ones used in~\cite{BCF12,BCFV15}.
Intermediate disorder scaling results were however proved more recently for the unperturbed multi-layer semi-discrete directed polymer in~\cite{Nic16} using the same ideas as in~\cite{QMR12}.

\begin{theorem}\label{ThmDiscrToContGen}
Fix $T>0$, $X\in\R$ and real vectors $b=(b_1,\dots,b_m)\in\R^m$ and $\beta=(\beta_1,\dots \beta_n)\in\R^n$ which satisfy \eqref{bbetacond}.
Set $\sigma=(2/T)^{1/3}$ and $\kappa=\sqrt{T/N}+X/N$ which yield by \eqref{parametr2} that $\tau=\kappa N=\sqrt{TN}+X$.
Let the drifts be given by $a=(a_1,\dots,a_m,0,\dots,0)\in\R^N$ where $a_l=\sqrt{N/T}+1/2+b_l$ for $1\le l\le m$ and the boundary parameters by $\alpha_k=\sqrt{N/T}+1/2+\beta_k$ for $1\le k\le n$.
Consider the O'Connell--Yor semi-discrete directed random polymer partition function $\Zsd^{a,\alpha}(\tau,N)$ defined in \eqref{defZOY} with parameters $a$ and $\alpha$.
With the scaling factor
\begin{equation}\label{eq:scalingFactorm}
C(N,m,T,X)=\exp\left(\frac12(N-m)\ln\left(\frac{T}{N}\right)+N+\frac12\left(\sqrt{TN}+X\right)+X\sqrt{\frac{N}{T}}\right),
\end{equation}
one has the convergence in distribution
\begin{equation}
\frac{\Zsdaalpha(\sqrt{TN}+X,N)}{C(N,m,T,X)}\Rightarrow\Zbbeta(T,X)
\end{equation}
as $N$ goes to infinity where $\Zbbeta(T,X)$ is the CDRP with $(m,n)$-spiked boundary perturbation given in \eqref{eq:SHEIC}.
\end{theorem}

The main contribution of this work gives the large time limit of the CDRP free energy with $(m,n)$-spiked boundary perturbation.
\begin{theorem}\label{ThmTimeLimitGen}
Let $b = (b_1,\dots,b_m)\in \R^m$ and $\beta=(\beta_1,\dots\beta_n)\in\R^n$ be such that $b_l<\beta_k$ for all $1\leq l \leq m$ and $1\leq k \leq n$.
Let $\sigma = (2/T)^{1/3}$ be scaled with the time parameter and let $Y\in\R$ and $r\in\R$ be arbitrary.
Then for the free energy of the CDRP with $(m,n)$-spiked boundary perturbation of parameters $\sigma b$ and $\sigma\beta$ at rescaled position $X=2^{1/3}YT^{2/3}$,
\begin{equation}\label{eq:timeLimitGen}
\lim_{T\to\infty} \PP\left( \frac{\Fsigbbeta(T,2^{1/3}YT^{2/3})+T/24}{(T/2)^{1/3}} \leq r \right) = F_{\BP,b+Y,\beta+Y} \left(r+Y^2\right)
\end{equation}
holds where $F_{\BP,b+Y,\beta+Y}$ is the Borodin--P\'ech\'e distribution function given by \eqref{defBPdistr} with parameter vectors shifted coordinatewise.
\end{theorem}

\section{Scaling limit for the O'Connell--Yor semi-discrete polymer}
\label{s:OYlimit}

We prove Theorem~\ref{ThmFreeEnergyScalingLimit} in this section which is a modification of the proof of Theorem 1.3 in~\cite{BCF12}.
We mention that in Theorem 1.3 in~\cite{BCF12} which is the $n=0$ case of Theorem~\ref{ThmFreeEnergyScalingLimit}, the factor $c_\kappa$ is missing in the scaling of parameters \eqref{aalphascaling}.
To keep our discussion self-contained, we recall the main steps of the proof and extend it to the present setup.

Let us scale
\begin{equation}\label{uscale}
u=u(N,r,\kappa) = \exp(-Nf_{\kappa} - rc_{\kappa}N^{1/3})
\end{equation}
and set $\tau = \kappa N$.
After the change of variables $\tilde{z} = s+v$ in \eqref{defKu} and by using Euler's reflection formula $1/(\Gamma(-s)\Gamma(1+s)) = -\pi / \sin(\pi s)$,
\begin{multline}\label{Kurewrite}
\K{u}(v,v')
=\frac{-1}{2\pi \I}\int_{\Cv{\tilde z}}\d \tilde{z}\frac{\pi}{\sin(\pi(\tilde{z}-v))}\frac{\exp(NG(v)+rc_{\kappa}N^{1/3}v)}{\exp(NG(\tilde{z})+rc_{\kappa}N^{1/3}\tilde{z})}\frac{1}{\tilde{z}-v'}\\
\times\prod_{l=1}^{m} \frac{\Gamma(v-a_l)\Gamma(\tilde{z})}{\Gamma(\tilde{z}-a_l)\Gamma(v)}
\prod_{k=1}^n\frac{\Gamma(\alpha_k-v-s)}{\Gamma(\alpha_k-v)}
\end{multline}
where $G(z)=\ln\Gamma(z)-\kappa z^2/2+f_{\kappa}z$.
The integration contour $\Cv{\tilde z}$ in \eqref{Kurewrite} was defined in~\cite{BCF12} in the absence of boundary parameters to be
\begin{equation}\label{defCztilde}
\left\{\theta_\kappa+\tilde{\varepsilon} + \I y, y\in\R \right\} \cup \bigcup_{q=1}^r B_{v+q}
\end{equation}
where $B_{v+q}$ denotes a small circle around $v+q$ and clockwise oriented.
$r\in\N_0$ is chosen such that $\Re(v)+r\le\theta_\kappa+\ordo(N^{-1/3})$ and we set $\tilde{\varepsilon} = p(v)c_\kappa^{-1}N^{-1/3}$ with $p(v)\in\{1,3\}$.
This choice of $r$ and $p(v)$ is needed to keep a uniformly positive distance from the poles coming from the sine in the denominator in \eqref{Kurewrite},
see Section 5.1 in~\cite{BCF12} for the precise definition.
It is also argued in~\cite{BCF12} that kernel $K_u$ has enough decay along the contour in \eqref{defCztilde}
which corresponds to the $\varphi=\pi/4$ case for the $\Cv{a;\alpha;\varphi}$ contour in Theorem~\ref{ThmFormulaSemiDiscrete}.

In the present setup when there are boundary parameters $a_l$ and $\alpha_k$ scaled according to \eqref{aalphascaling},
the contour $\Cv{\tilde z}$ is defined to be the contour \eqref{defCztilde} with a local modification in an $N^{-1/3}$ neighbourhood of $\theta_\kappa$
in a way that it crosses the real axis between the $a_l$ and the $\alpha_k$ singularities.
By the Cauchy theorem, the contour $v+\Cs{v}$ for $\tilde z$ seen on the left of Figure~\ref{FigContoursSemiDiscrete} can be replaced by $\Cv{\tilde z}$
without changing the kernel $\K{u}$ in \eqref{Kurewrite}.

The function $G$ has a double critical point at $\theta_\kappa$, i.e.\ $G(v) \simeq G(\theta_\kappa) - \frac{(c_\kappa)^3}{3}(v-\theta_\kappa)^3$.
This suggests the rescaling around $\theta_\kappa$ by $N^{1/3}$, that is the change of variables
\begin{equation}\label{Phichange}
\left\{ v,v',\tilde{z} \right\} = \left\{ \Phi(w), \Phi(w'), \Phi(z) \right\} \quad \textrm{with} \quad \Phi(z) = \theta_\kappa + zc_\kappa^{-1}N^{-1/3}.
\end{equation}
Then the rescaled kernel is defined as
\begin{multline}\label{eq:rescaledKernel}
\K{N}(w,w') = c_\kappa^{-1}N^{-1/3}\K{u}(\Phi(w), \Phi(w'))
= \frac{-c_\kappa^{-1}N^{-1/3}}{2\pi\I} \int_{\Phi^{-1}(\Cv{\tilde z})} \d z \frac{\pi e^{NG(\Phi(w))-NG(\Phi(z))}}{\sin(\pi(z-w)c_\kappa^{-1}N^{-1/3})}\\
\times\frac{e^{r(w-z)}}{z-w'}
\prod_{l=1}^m\frac{\Gamma(\Phi(w)-a_l)\Gamma(\Phi(z))}{\Gamma(\Phi(z)-a_l)\Gamma(\Phi(w))}
\prod_{k=1}^n\frac{\Gamma(\alpha_k-\Phi(z))}{\Gamma(\alpha_k-\Phi(w))}.
\end{multline}

Let the new contour $\Cv{w}$ be the local perturbation of $\left\{-|y|+\I y, y\in\R \right\}$ in a constant neighbourhood of $0$
in a way that it crosses the real axis between the $b_l$ and $\beta_k$ singularities as shown on Figure~\ref{fig:contour_perturb}, also compare with Figure~\ref{FigContoursSemiDiscrete}.
Further, let $\Cv{z}$ be the local modification of $1+\I\R$ in a neighbourhood of $0$ so that it does not intersect $\Cv{w}$ and it crosses the real axis between the two families of singularities.
Then one can replace $\Cv{a;\alpha;\varphi}$ by $\Cv{w}$ and the integration path $\Phi^{-1}(\Cv{\tilde z})$ in \eqref{eq:rescaledKernel} by $\Cv{z}$ so that
one has the equality of Fredholm determinants $\fredholmp{\K{u}}{L^2(\Cv{a;\alpha;\varphi})} = \fredholmp{\K{N}}{L^2(\Cv{w})}$ by the Cauchy theorem.

\begin{figure}\begin{center}
\psfrag{Cw}{$\Cv{w}$}
\psfrag{Cz}{$\Cv{z}$}
\psfrag{0}{$0$}
\psfrag{b}{$b$'s}
\psfrag{beta}{$\beta$'s}
\includegraphics[height=4cm]{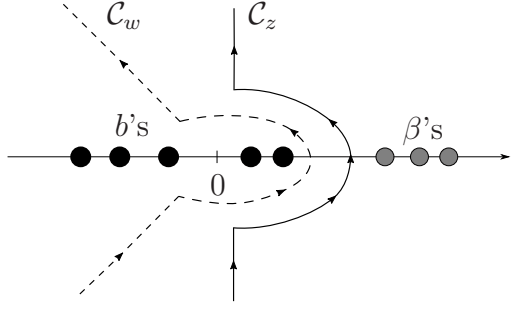}
\caption{Integration paths $\Cv{w}$ and $\Cv{z}$.
The black dots on the left are the values of $b_1,\dots,b_m$ and the grey dots on the right are $\beta_1,\dots,\beta_n$.\label{fig:contour_perturb}}
\end{center}\end{figure}

Based on the next two propositions and by using Lemma~\ref{lemmaConv} below, Theorem~\ref{ThmFreeEnergyScalingLimit} on the scaling limit for the O'Connell--Yor semi-discrete polymer can be verified.

\begin{proposition}\label{propPointwiseLim}
Let $\K{N}(w,w')$ be given in \eqref{eq:rescaledKernel}.
Uniformly for $w,w'$ in a bounded set of $\Cv{w}$,
\begin{equation}\label{kernelpointwise}
\lim_{N\to\infty}\K{N}(w,w')= \wt{\K{}}_{\BP,b,\beta}(w,w')
\end{equation}
where
\begin{equation}\label{limKernelTilde}
\wt{\K{}}_{\BP,b,\beta}(w,w'):=\frac{1}{2\pi\I}\int_{\Cv{z}} \d z \frac{1}{(w-z)(z-w')} \frac{e^{z^3/3-rz}}{e^{w^3/3-rw}}
\prod_{l=1}^{m}\frac{z - b_l}{w - b_l}
\prod_{k=1}^{n}\frac{w - \beta_k}{z -\beta_k}.
\end{equation}
\end{proposition}

\begin{proposition}\label{propKernelBound}
For any $w,w'\in\Cv{w}$ there exists a constant $C\in (0,\infty)$ such that 
\begin{equation}\label{kernelbound}
|\K{N}(w,w')| \leq C e^{-|\Im(w)|}
\end{equation}
uniformly for all $N$ large enough.
\end{proposition}

\begin{lemma}\label{lemmaConv}\cite[Lemma 4.1.38]{BC11}
Consider a sequence of functions $\left(\Theta_n \right)_{n\geq 1}$  mapping $\R \to \left[ 0,1 \right]$ with the following properties:
$x\mapsto\Theta_n(x)$ is strictly decreasing, $\lim_{x\to -\infty} \Theta_n(x) = 1$, $\lim_{x\to \infty} \Theta_n(x) = 0$ for all $n$
and $\Theta_n(x)\to\Id_{x\leq 0}$ as $n\to\infty$ uniformly on $\R\setminus\left[-\delta,\delta\right]$ for all $\delta>0$.
Consider a sequence of random variables $X_n$ and a continuous probability distribution function $p(r)$ such that $\EE\left[\Theta_n(X_n - r)\right]\to p(r)$ as $n\to\infty$ for each $r\in\R$.
Then $X_n$ converges in distribution to the distribution given by $p(r)$.
\end{lemma}

\begin{proof}[Proof of Theorem~\ref{ThmFreeEnergyScalingLimit}]
By Hadamard's bound and by dominated convergence, Proposition~\ref{propPointwiseLim} and \ref{propKernelBound} together imply that
\begin{equation}\label{fredholmconv}
\fredholmp{\K{N}}{L^2(\Cv{w})} \to \fredholmm{\wt{\K{}}_{\BP,b,\beta}}{L^2(\Cv{w})}=\fredholmm{\K{\BP,b,\beta}}{L^2((r,\infty))}=F_{\BP,b,\beta}(r)
\end{equation}
as $N\to\infty$ where the first equality above follows from the same reformulation of Fredholm determinants as in Lemma 8.7 of~\cite{BCF12}.

Let us define a sequence of functions $\Theta_N(x) = \exp(-\exp(c_{\kappa}N^{1/3}x))$.
Now by \eqref{uscale},
\begin{equation}
\EE\left[ \Theta_N\left( \frac{F^{a, \alpha} (\kappa N) - N f^{\kappa}}{c_{\kappa}N^{1/3}} - r\right)\right] = \EE\left[ e^{u\Zsdaalpha(\tau,N)} \right] =\fredholmp{\K{u}}{\LC} \to F_{\BP,b,\beta}(r)
\end{equation}
as $N\to\infty$ where we used the definition of $\Theta_N$, Theorem~\ref{ThmFormulaSemiDiscrete} and \eqref{fredholmconv}.
To conclude the proof, one uses Lemma~\ref{lemmaConv} with $p(r) = F_{\BP,b,\beta}(r)$.
\end{proof}

We introduce the extra gamma factors
\begin{equation}\label{defPQ}
P(w,z,a)=\frac{\Gamma(\Phi(w)-a_l)}{\Gamma(\Phi(w))}\frac{\Gamma(\Phi(z))}{\Gamma(\Phi(z)-a_l)},\qquad
Q(w,z,\alpha_k)=\frac{\Gamma(\alpha_k-\Phi(z))}{\Gamma(\alpha_k-\Phi(w))}
\end{equation}
for $1\le l\le m$ and $1\le k\le n$.
To extend the proofs of Proposition 5.1 and 5.2 of~\cite{BCF12} to those of Propostion~\ref{propPointwiseLim} and \ref{propKernelBound}, the following lemma about the bounds on the extra factors is the key.

\begin{lemma}\label{lemma:Qbound}
Once the contours $\Cv{w}$ and $\Cv{z}$ are fixed, there is a constant $C$ such that
\begin{equation}\label{Qbound}
|Q(w,z,\alpha_k)|\le C\frac{|w|}{N^{1/3}}\left(1+\frac{N^{1/3}}{|z|}\right)\le C|w|\left(1+\frac1{|z|}\right)
\end{equation}
as long as $w\in\Cv{w}$ and $z\in\Cv{z}$.

Furthermore, let $N$ be large enough to make the $N^{-1/3}$ difference of $\Cv{\tilde z}$ and the contour in \eqref{defCztilde} small.
Then the small circles in $\Cv{\tilde z}$ and in \eqref{defCztilde} can only be present, i.e.\ $r>0$ can only happen for a $v\in\Cv{a;\alpha;\varphi}$ if $|v|>\varepsilon$ for some fixed $\varepsilon>0$.
In this case there is a $C$ such that for any $v\in\Cv{a;\alpha;\varphi}$ and $q=1,\dots,r$,
\begin{equation}\label{residuebound}
|Q(\Phi^{-1}(v),\Phi^{-1}(v+q),\alpha_k)|\le C.
\end{equation}
\end{lemma}

\begin{proof}
After substituting \eqref{Phichange} and the scaling \eqref{aalphascaling} into the definition \eqref{defPQ}, one can write
\begin{equation}\label{Qrewrite}
Q(w,z,\alpha)=\frac{\Gamma\left((\beta_k-z+o(1))c_\kappa^{-1}N^{-1/3}\right)}{\Gamma\left((\beta_k-w+o(1))c_\kappa^{-1}N^{-1/3}\right)}.
\end{equation}
First we show that for the numerator
\begin{equation}\label{Qnumer}
\left|\Gamma\left((\beta_k-z+o(1))c_\kappa^{-1}N^{-1/3}\right)\right|\le C\left(1+\frac{N^{1/3}}{|z|}\right)
\end{equation}
holds if $z\in\Cv{z}$.
To this end, we use the asymptotics
\begin{equation}\label{eq:AsympGammaIm}
\lim_{|y|\to\infty} |\Gamma(x+iy)|(2\pi)^{-1/2}e^{\frac12 \pi|y|}|y|^{\frac12 - x} = 1
\end{equation}
from equation 6.1.45 of~\cite{AS84}.
If $z\in\Cv{z}$ and $|z|>\delta N^{1/3}$ for some fixed $\delta>0$, then the real part of the argument of the gamma function in \eqref{Qnumer} goes to $0$ as $N\to\infty$, hence it is bounded.
Consequently, \eqref{eq:AsympGammaIm} yields an exponential decay of $|\Gamma\left((\beta_k-z+o(1))c_\kappa^{-1}N^{-1/3}\right)|$ in $|z|$ which we bound by a constant.
By the asymptotics $\Gamma(Z)\sim1/Z$ around $Z=0$, one gets that the left-hand side of \eqref{Qnumer} can be upper bounded by $CN^{1/3}/|z|$ as long as $|z|<\delta N^{1/3}$.
This proves \eqref{Qnumer}.

Next we prove for the denominator that
\begin{equation}\label{Qdenum}
\left|\Gamma\left((\beta_k-w+o(1))c_\kappa^{-1}N^{-1/3}\right)\right|\ge c\left(1+\frac{N^{1/3}}{|w|}\right)
\end{equation}
for $w\in\Cv{w}$ with a constant $c$ small enough. Equation 6.1.37 in~\cite{AS84} reads as
\begin{equation}\label{eq:AsympGamma}
\Gamma(z) = e^{-z}z^{z-\frac12}(2\pi)^{1/2}\left(1+\ordo\left(\frac{1}{z}\right)\right).
\end{equation}
For $w\in\Cv{w}$ and $|w|>\delta N^{1/3}$, one can write $w=-tN^{1/3}\pm\I tN^{1/3}$ for some $t>\delta/\sqrt2$.
Hence the left-hand side of \eqref{Qdenum} grows as $Ce^{-t}t^t$ as $t\to\infty$ which we can lower bound by a small constant as $t>\delta/\sqrt2$.
If $|w|<\delta N^{1/3}$, by the asymptotics $\Gamma(Z)\sim1/Z$ around $Z=0$ again, the left-hand side of \eqref{Qdenum} is lower bounded by $cN^{1/3}/|w|$.
This shows \eqref{Qdenum}.
Putting \eqref{Qnumer} and \eqref{Qdenum} together yields \eqref{Qbound} with a large enough $C$.

The uniform lower bound on $|v|$ follows from the choice of the contours.
On the one hand, $r$ is chosen such that $\Re(v)+r\le\theta_\kappa+\ordo(N^{-1/3})$.
On the other hand, $v\in\Cv{a;\alpha;\varphi}$ satisfies $v=\theta_\kappa+\ordo(N^{-1/3})+e^{\I(\pi\pm\varphi)}y$ for some $y\in\R_+$.
These two properties imply the lower bound if the small circles are present.

To show \eqref{residuebound} if the circles are present, observe that the ratio which we want to bound in absolute value simplifies as
\begin{equation}
Q(\Phi^{-1}(v),\Phi^{-1}(v+q),\alpha_k)=\frac{\Gamma(\alpha_k-v-q)}{\Gamma(\alpha_k-v)}=\frac{1}{(\alpha_k-v-1)\dots(\alpha_k-v-q)}.
\end{equation}
This is bounded by an absolute constant since $|v|>\varepsilon$ also means that $|\Im(v)|$ is uniformly positive.
\end{proof}

\begin{proof}[Proof of Proposition~\ref{propPointwiseLim}]
Knowing the asymptotics $\Gamma(z) \simeq 1/z$ near zero from \cite{AS84}, one can conclude from \eqref{Qrewrite} that under the scaling \eqref{aalphascaling},
$Q(w,z,\alpha_k)\to(w-\beta_k)/(z-\beta_k)$ holds as $N\to\infty$ for $k=1,\dots,n$.
Similarly, $P(w,z,a_l)\to(z-b_l)/(w-b_l)$ for $l=1,\dots,m$.
As in the proof of Proposition 5.1 in~\cite{BCF12}, the Taylor expansion of the remaining factors in the integrand of $\K{N}$ in \eqref{eq:rescaledKernel} yields that
the integrand converges to the integrand of $\wt{\K{}}_{\BP,b,\beta}$ in \eqref{limKernelTilde}.

One can apply dominated convergence as it is done in the proof of Proposition 5.1 in~\cite{BCF12}.
It was proved in~\cite{BCF12} based on Lemma 5.4 that the integration contour of $\K{N}$ in $z$ is steep descent for the function $-\Re(G(\Phi(z)))$
with derivative going to $-\infty$ linearly in $|\Im(\tilde z)|=N^{-1/3}|\Im(z)|$.
Since the $Q$ factors are bounded in \eqref{Qbound},
the decay of $e^{-NG(\Phi(z))}$ ensures that the integral which defines the kernel $K_N$ in \eqref{eq:rescaledKernel} is still convergent in the presence of the $Q$ factors.
Hence the steps of the proof of Proposition 5.1 in~\cite{BCF12} can be followed.
In particular, the integral which defines $\K{N}$ restricted to the set $|\Im(z)|>\delta N^{1/3}$ is $\ordo(e^{-c(\delta)N})$.
On the other hand on $|\Im(z)|<\delta N^{1/3}$, one can replace the integrand of $\K{N}$ by the integrand of $\wt{\K{}}_{\BP,b,\beta}$ with an overall error of order $\ordo(N^{-1/3})$.
This verifies the convergence of the kernels \eqref{kernelpointwise}.
\end{proof}

\begin{proof}[Proof of Proposition~\ref{propKernelBound}]
The exponential bounds obtained in the proof of Proposition 5.2 in~\cite{BCF12} are not affected by the presence of extra polynomial factors which upper bound $Q$ in \eqref{Qbound}--\eqref{residuebound}.
Hence \eqref{kernelbound} follows.
\end{proof}

\section{Large time limit of the CDRP with $(m,n)$-spiked boundary}
\label{s:CDRPlimit}

In this section, we prove Theorem~\ref{ThmTimeLimitGen} about the large time limit of the free energy $\Zbbeta$ of the CDRP with $(m,n)$-spiked boundary perturbations.
We start by giving a Fredholm determinant formula for its Laplace transform below in Proposition~\ref{PropLaplaceFredholm} based on Theorem~\ref{ThmFormulaSemiDiscrete}.
Let $b=(b_1,\dots,b_m)$ and $\beta=(\beta_1,\dots,\beta_n)$ be such that \eqref{bbetacond} holds.
Define the kernel
\begin{equation}\label{defKbbeta}
\Kbbeta(x,y) = \frac1{(2\pi\I)^2}\int\d w\int\d z \frac{\sigma\pi S^{\sigma(z-w)}}{\sin(\sigma\pi(z-w))} \frac{e^{z^3/3-zy}}{e^{w^3/3-wx}}
\prod_{l=1}^{m} \frac{\Gamma(\sigma w-b_l)}{\Gamma(\sigma z-b_l)}\prod_{k=1}^{n}\frac{\Gamma(\beta_k-\sigma z)}{\Gamma(\beta_k-\sigma w)}
\end{equation}
where
\begin{equation}\label{defsigma}
\sigma = \left(2/T\right)^{1/3}
\end{equation}
and the integration contour for $w$ is from $-\frac1{4\sigma}-\I\infty$ to $-\frac1{4\sigma}+\I\infty$ and crosses the real axis between $\max_{1\le l\le m}b_l/\sigma$ and $\min_{1\le k\le n}\beta_k/\sigma$.
The other contour for $z$ goes from $\frac1{4\sigma}-\I\infty$ to $\frac1{4\sigma}+\I\infty$,
it also crosses the real axis between $\max_{1\le l\le m}b_l/\sigma$ and $\min_{1\le k\le n}\beta_k/\sigma$ and it does not intersect the contour for $w$.

\begin{proposition}\label{PropLaplaceFredholm}
Fix $S$ with positive real part, $T>0$, $b$ and $\beta$ real vectors with \eqref{bbetacond}.
Set $\sigma$ as in \eqref{defsigma}.
Then
\begin{equation}\label{LaplaceFredholm}
\EE\left[ \exp\left(-Se^{\frac{X^2}{2T}+\frac T{24}}\Zbbeta(T,X)\right)\right] = \fredholmm{\K{b+X/T,\beta+X/T}^{(\sigma)}}{\LRplus}
\end{equation}
where $\Zbbeta$ is the partition function of the CDRP with $(m,n)$-spiked boundary perturbations and $\Kbbeta$ is defined in \eqref{defKbbeta}.
\end{proposition}

\begin{proof}
Let Theorem~\ref{ThmFormulaSemiDiscrete} be used with
\begin{equation}\label{defu}
u=\frac S{C(N,m,T,X)}e^{\frac{X^2}{2T} + \frac{T}{24}}
\end{equation}
where $C(N,m,T,X)$ is given by \eqref{eq:scalingFactorm}.
Then on the left-hand side of \eqref{semidiscreteFredholm} with $\tau=\sqrt{TN}+X$, Theorem~\ref{ThmDiscrToContGen} on the intermediate disorder scaling yields the convergence in distribution
\begin{equation}\label{uZconvdistr}
u\Zsdaalpha(\sqrt{TN}+X,N)\Rightarrow Se^{\frac{X^2}{2T}+\frac T{24}}\Zbbeta(T,X)
\end{equation}
as $N\to\infty$.
By definition \eqref{defZOY}, the partition function $\Zsdaalpha$ is positive, hence \eqref{uZconvdistr} implies the convergence of the Laplace transforms
\begin{equation}
\EE\left[  e^{-u\Zsdaalpha(\tau,N)} \right] \to \EE\left[ \exp\left( -Se^{\frac{X^2}{2T}+\frac T{24}} \Zbbeta(T,X) \right) \right]
\end{equation}
as $N\to\infty$ where $\tau=\sqrt{TN}+X$ and $u$ is defined in \eqref{defu}.

On the other hand, the same scaling of parameters is used on the right-hand side of \eqref{semidiscreteFredholm}.
Then Theorem 6.3 of~\cite{BCFV15} is used to conclude the convergence of Fredholm determinants
\begin{equation}
\lim_{N\to\infty} \det\left(\Id+ \K{u}\right)_{L^2(\Cv{a;\alpha;\pi/4})}=\det\left(\Id-\K{b+X/T,\beta+X/T}^{(\sigma)}\right)_{\LRplus}
\end{equation}
under the following scaling of the parameters.
As in Theorem~\ref{ThmDiscrToContGen}, one sets $\tau=\sqrt{TN}+X$, $\kappa=\tau/N$ and $\theta_\kappa$ is given by \eqref{parametr2}.
This means that $\theta_\kappa=\sqrt{N/T}-X/T+1/2+\ordo(N^{-1/2})$.
One sets $u$ given by \eqref{defu} and $\sigma$ given by \eqref{defsigma}.
For the boundary parameters $a_l$ and $\alpha_k$, instead of the scaling given in Theorem~\ref{ThmDiscrToContGen},
one sets $a_l=\theta_\kappa+b_l$ and $\alpha_k=\theta_\kappa+\beta_k$ according to Section 6 of~\cite{BCFV15}.
This difference results in the shift by $X/T$ in the rescaled boundary parameters $b_l$ and $\beta_k$ which completes the proof.
\end{proof}

The following proposition is the key for the proof of Theorem~\ref{ThmTimeLimitGen}.

\begin{proposition}\label{propFredholmKonv}
We have
\begin{equation}\label{sigmafredholmconv}
\fredholmm{\Ksigbbeta}{\LRplus} \to \fredholmm{\K{BP,b,\beta}}{\Lrinf}
\end{equation}
as $\sigma\to0$	where $\Kbbeta$ and $\K{BP,b,\beta}$ are given in \eqref{defKbbeta} and \eqref{defBPkernel}.
\end{proposition}

\begin{proof}[Proof of Theorem~\ref{ThmTimeLimitGen}]
Let $S = e^{-r/\sigma}$ and define the functions $\Theta_T(x) = \exp(-e^{x/ \sigma})$ where $\sigma=(2/T)^{1/3}$.
Observe that one can write
\begin{equation}\label{eq:theta}
\exp\left(-Se^{\frac{X^2}{2T}+\frac T{24}}\Zbbeta(T,X)\right)=\Theta_T\left(\frac{\Fsigbbeta(T,X)+\frac{X^2}{2T}+\frac T{24}}{\sigma^{-1}}-r\right).
\end{equation}
By taking expectation above
\begin{equation}\label{ETheta}\begin{aligned}
\EE\left[\Theta_T\left(\frac{\Fsigbbeta(T,X)+\frac{X^2}{2T}+\frac T{24}}{\sigma^{-1}}-r\right)\right]&=\EE\left[ \exp\left(-Se^{\frac{X^2}{2T}+\frac T{24}}\Zbbeta(T,X)\right) \right]\\
&=\fredholmm{\K{\sigma b+X/T,\sigma\beta+X/T}^{(\sigma)}}{\LRplus}\\
&\to\fredholmm{\K{BP,b+Y,\beta+Y}}{\Lrinf}
\end{aligned}\end{equation}
as $T\to\infty$ where we used Proposition~\ref{PropLaplaceFredholm} in the second equation above.
To conclude the convergence in \eqref{ETheta}, Proposition~\ref{propFredholmKonv} was used with boundary parameters $\sigma b+X/T=\sigma(b+Y)$ and $\sigma\beta+X/T=\sigma(\beta+Y)$ where $X=2^{1/3}YT^{2/3}$.

The functions $\Theta_T$ satisfy the properties of Lemma~\ref{lemmaConv},
hence by \eqref{ETheta} the lemma is applicable for the random variables $\sigma(\Fsigbbeta(T,X)+X^2/(2T)+T/24)$ and with the distribution function $F_{\BP,b+Y,\beta+Y}(r)$ defined by \eqref{defBPdistr}.
By observing that $\sigma X^2/(2T)=Y^2$ and by substituting $r$ by $r+Y^2$, one arrives to \eqref{eq:timeLimitGen}.
\end{proof}

We are left with proving Proposition~\ref{propFredholmKonv}. We use the following decay bound from~\cite{BCFV15} adapted to the present setting.

\begin{lemma}\label{lemmaKernelEstGen}\cite[Lemma~B.4]{BCFV15}
Fix $b_1\leq b_2 \leq \dots \leq b_m < \beta_1 \leq \beta_2 \dots \leq \beta_n$ so that $\beta_i - b_j < 1$ for any $1\leq i \leq n$ and $1\leq j \leq m$.
Then there is a finite constant $C$ such that for any $x,y\in\Rplus$
\begin{equation}
\left|\Kbbeta(x,y)\right| \leq C\exp\left( -\frac{\beta_1}{\sigma}y + \frac{b_m}{\sigma}x \right).
\end{equation}
\end{lemma}

\begin{proof}[Proof of Proposition~\ref{propFredholmKonv}]
By setting $S=e^{-r/\sigma}$, the kernel on the left-hand side of \eqref{sigmafredholmconv} reads as
\begin{equation}\label{kernelsigma}
\Ksigbbeta(x,y)=\frac1{(2\pi\I)^2}\int\d w\int\d z \frac{\sigma\pi e^{-r(z-w)}}{\sin(\sigma\pi(z-w))} \frac{e^{z^3/3-zy}}{e^{w^3/3-wx}}
\prod_{l=1}^{m} \frac{\Gamma(\sigma(w-b_l))}{\Gamma(\sigma(z-b_l))}\prod_{k=1}^{n}\frac{\Gamma(\sigma(\beta_k-z))}{\Gamma(\sigma(\beta_k-w))}.
\end{equation}
Then the first factor in the double integral in \eqref{kernelsigma} converges to $e^{-r(z-w)}/(z-w)$ as $\sigma\to0$.
For the product of the gamma ratios,
\begin{equation}
\prod_{l=1}^m\frac{\Gamma(\sigma(w-b_l))}{\Gamma(\sigma(z-b_l))}\prod_{k=1}^n\frac{\Gamma(\sigma(\beta_k-z))}{\Gamma(\sigma(\beta_k-w))}
\to\prod_{l=1}^m\frac{z-b_l}{w-b_l}\prod_{k=1}^n\frac{w-\beta_k}{z-\beta_k}
\end{equation}
as $\sigma\to0$.
Hence the integrand in \eqref{kernelsigma} converges to that of $\K{\BP,b,\beta}(x+r,y+r)$ given in \eqref{defBPkernel} as $\sigma\to0$.
Since along the contours for $w$ and $z$, the factors $e^{z^3/3-w^3/3}$ in \eqref{kernelsigma} have fast enough decay, we conclude that
\begin{equation}\label{kernelConvBP}
\lim_{\sigma\to0}\Ksigbbeta(x,y)=\K{\BP,b,\beta}(x+r,y+r).
\end{equation}

To show that the convergence of the kernels \eqref{kernelConvBP} implies the convergence of Fredholm determinants \eqref{sigmafredholmconv}, one uses dominated convergence.
Lemma~\ref{lemmaKernelEstGen} applied to $\Ksigbbeta$ provides a uniform upper bound in $\sigma$.
Using this upper bound, the $n$th term in the Fredholm determinant expansion of the left-hand side of \eqref{sigmafredholmconv} is bounded by
\begin{equation}\label{fredholmbound}\begin{aligned}
&\frac{1}{n!} \int_{\Rplus} \dots \int_{\Rplus} \det\left[\Ksigbbeta(x_i,x_j)\right]_{i,j=1}^{n} \d x_1\dots\d x_n\\
&\qquad\leq\frac{C^{2n}n^{n/2}}{n!} \int_{\Rplus} \dots \int_{\Rplus} e^{-(\beta_1-b_m)\sum_{j=1}^{n} x_j} \d x_1\dots\d x_n\\
&\qquad=\frac{C^{2n}n^{n/2}}{(\beta_1-b_m)^nn!}
\end{aligned}\end{equation}
where we also used the Hadamard bound in the first inequality above.
Since the right-hand side of \eqref{fredholmbound} is summable, dominated convergence implies \eqref{sigmafredholmconv}.
\end{proof}


\end{document}